\documentclass[12pt, leqno]{amsart}

\usepackage{float}
\usepackage{graphicx}
\usepackage[margin=1in]{geometry}

\theoremstyle{theorem}
\newtheorem{theorem}{Theorem}
\newtheorem{corollary}{Corollary}
\newtheorem*{remark}{Remark}
\newtheorem*{acknowledgment}{Acknowledgments}
\def\Z{\Bbb{Z}}
\def\R{\mathbb{R}}

\def\C{\mathbb{C}}
\def\sinc{{\rm sinc}}
\def\Si{{\rm Si}}
\def\rect{{\rm rect}}

\def \sech{{\rm sech}}
\begin{document}

\title{The Binomial Coefficient as an (In)finite Sum of Sinc Functions}
\markright{Binomial  coefficient via  damped sine waves}
\author{Lorenzo David}

\maketitle
\vspace{-11truemm}
\begin{abstract}
In this article, we give a formula for the generalization of the binomial coefficient to the complex numbers as a linear combination of $\sinc$ functions. We then give a general formula to compute the integral on the real line of the product of the binomial coefficient and a given function, which, in some cases, turns out to be equal to the series of their values on the integers. Finally, we establish a list of identities obtained by applying these formulas.
\end{abstract}
\vspace{-5truemm}

\section{Introduction.}

The binomial coefficient  $\binom m k=\frac{m!}{k!(m-k)!}$ is the number of ways of picking $k$ unordered outcomes from $m$ possibilities, also known as a combination. For any complex number $w$ and integer $k$, the definition is extended as follows:
\begin{equation}\label{BinomDef}
 \binom w k =   \begin{cases}
                                  \displaystyle\frac{w(w-1)\cdots(w-k+1)}{k!}, & \text{if $k\geq 1$}; \\
                                    1, & \text{if $k=0$};\\
                                    0, & \text{otherwise}. \\
  \end{cases}
\end{equation}  
Moreover, the gamma function allows the binomial coefficient to be generalized to two complex arguments as
 \begin{equation}\label{GammaDef}  \binom w z  = \frac {\Gamma(w+1)}{\Gamma(z+1) \Gamma(w-z+1)}, \quad w \not\in\Z_{\le-1}.
\end{equation}
For $\mathfrak{R}(w)>-1$, the graph of the real and of the imaginary part of $\binom w z$  resembles a damped sinusoid; see Figure \ref{Figure1}. This suggests a possible relation to $\sinc(z)$, which we will prove as one of our main results.
\begin{figure}[H]
\centering{\includegraphics[scale=0.45]{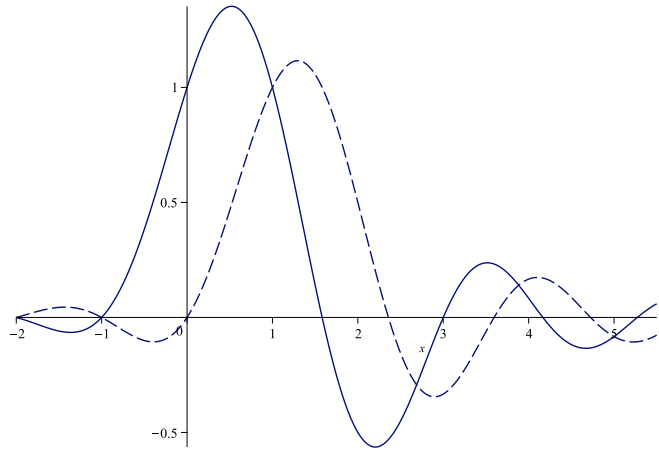}}
\caption{A plot of the real part (thick line) and imaginary part (dashed line) of $\binom {1+i}{x}$.}
\label{Figure1}
\end{figure}
\vskip 1 cm 

\section{Preliminaries.}
We define the reciprocal gamma function, for any complex number $z$, by the Weierstrass form (see, for instance, {\cite[p. 57]{Hav}}):
\begin{equation}\label{RecGammaDef}
\frac{1}{\Gamma(z)}=ze^{\gamma z} \prod_{k=1}^{\infty}\left(1+\frac{z}{k}\right)e^{-z/k},
\end{equation}
where $\gamma$ is the Euler--Mascheroni constant.

From this definition one can deduce Euler's reflection formula (a detailed proof is given in {\cite[pp. 58--59]{Hav}}):
\begin{equation}\label{EulRef}
\frac{1}{\Gamma(z)\Gamma(1-z)}= \frac{\sin(\pi z)}{\pi}, \thickspace \thickspace \thickspace z \in \C.
\end{equation}
We will also make use of the generalized binomial theorem (for a proof see {\cite{Abel}}):
\begin{equation}\label{BinTheor}\sum_{k=0}^\infty \binom w k z^k=  
 \begin{cases}
                                  (1+z)^w, & \text{if $|z|<1$ and $w\in\C$}; \\
                                    & \text{or $|z|>1$ and $w\in\Z_{\geq 0}$}; \\
                                    & \text{or $|z|=1$, $z\neq -1$ and $\mathfrak{R}(w)>-1$}; \\
                                            0, &\text{if $z=-1$ and $\mathfrak{R}(w)>0$}. \\
  \end{cases}
\end{equation}
Another related function we will use in this article is the beta function:
\begin{equation} \label{BetaFunc} B(p,q)=\sum_{k=0}^{\infty}\binom{p-1}{k}\frac{(-1)^k}{q+k},\thickspace \thickspace \thickspace \mathfrak{R}(p)>0,\ q\in\C, \ q\notin \Z_{\leq0}.
\end{equation}
A straightforward proof that this definition is equivalent to the more familiar integral definition $ B(p,q)= \int_{0}^{1} t^{q-1}(1-t)^{p-1}dt$ can be given by expanding $(1-t)^{p-1}$ with (\ref{BinTheor}) and integrating. Moreover, since the beta function is holomorphic, we can extend this definition to the domain of analyticity of the series. (For more details on analytic continuation see {\cite[p. 234]{Fla}}.)

The beta function can also be defined from the gamma function by
\begin{equation}\label{BetaFunc2} 
B(p,q)= \frac{\Gamma(p)\Gamma(q)}{\Gamma(p+q)},\thickspace \thickspace \thickspace p,q \in \C,\thickspace p,q \notin \Z_{\leq 0}.
\end{equation}
For $f\in L^1(\R)+L^2(\R)$, we define the Fourier transform as \\
\begin{equation}\label{1FourierDef}
\widehat f(\xi)=\int_{-\infty}^{\infty}f(x) e^{-2\pi i x\xi}\,dx.
\end{equation}
If $f \in L^2(\R)$, the expression above is to be understood as an improper integral. 
Furthermore, we recall that a function having a Fourier transform with bounded support $[-b,b]$, is said bandlimited, with bandwidth $b$.

We will also make use of the function
\begin{equation}\label{SincDef}
\sinc(z):= \begin{cases}
                                  1, & \text{if $z=0$}; \\
                                           \displaystyle \frac{ \sin(\pi z)}{\pi z}, &\text{otherwise} \\
  \end{cases}
\end{equation}
and of its unnormalized integral
\begin{equation}\label{SiDef}
\Si(z):=\int_{0}^{z} \frac{ \sin t}{t}\,dt.
\end{equation}
As can be seen from the Fourier inversion theorem, the Fourier transform of $\sinc(x)$ is given by
\begin{equation}\label{RectDef}
\rect(x)=
  \begin{cases}
                                  1, & \text{if $|x|<1/2$}; \\
                                    0, & \text{otherwise}. \\
  \end{cases}
\end{equation}

We will use an important result regarding the $\sinc$ function, which, for $f$ in $L^1(\R)$, directly follows from the Fourier inversion theorem and from the fact that $\rect(x)\in L^1(\R)$, whereas, in $L^2(\R)$, it is a consequence of Plancherel theorem, namely
\begin{equation}\label{Plancherel} \int_{-\infty}^{\infty}\! f(x)\, \sinc(x-a)\, dx = \int_{-\frac{1}{2}}^{\frac{1}{2}} \widehat  f(\xi) e^{2\pi i \xi a}\,d\xi, \thickspace \thickspace \thickspace a \in \R.
\end{equation}

\section{Primary Results.}

\begin{theorem}\label{Theorem 1} For every nonnegative integer  $m$, the binomial coefficient $\binom{m}{z}$, where $z$ is a complex number, is equal to the following finite sum:
 \begin{equation}  \binom m z=\sum_{k=0}^m \binom m k  \sinc(z-k).
 \end{equation}
\end{theorem}

 \begin{theorem}\label{Theorem 2} Let $w$ be a complex number with real part greater than $-1$, and let $z$ be a complex number. Then the binomial coefficient  $\binom{w}{z}$ is equal to the following series of functions:
 \begin{equation}  \binom w z=\sum_{k=0}^\infty \binom w k  \sinc(z-k).
 \end{equation}
Moreover, the convergence is absolute.
\end{theorem}

\begin{theorem}\label{Theorem 3} If $f:\R\rightarrow \C$ is in either $L^1(\R)$ or $L^2(\R)$, then
\begin{equation} \int_{-\infty}^{\infty}\! \binom w x f(x) \,dx=\sum_{k=0}^{\infty}\binom w k \displaystyle \int_{-\frac{1}{2}}^{\frac{1}{2}}\! \widehat f(\xi) e^{2\pi i\xi k}\,d\xi  ,\thickspace \thickspace\thickspace\mathfrak{R}(w)>-1.
\end{equation}
\end{theorem}
\begin{corollary}\label{Corollary}If $g$ is bandlimited with bandwidth $\frac{1}{2}$, then

$$\displaystyle \int_{-\infty}^{\infty}\! \binom w x g(x) \,dx= \sum_{k=0}^{\infty}\binom w k g(k), \thickspace \thickspace\thickspace\mathfrak{R}(w)>-1.$$

\end{corollary}
\begin{remark}\rm In {\cite{Ram}}, it is proved that the binomial coefficient $\binom w x$ has bandwidth $\frac{1}{2}$; thus, the last corollary allows one to compute several integrals involving the binomial coefficient (see {\cite[Theorems 4, 6--8]{Sal}} for examples).
\end{remark}

\section{Proofs.}

\begin{proof}[Proof of Theorem \ref{Theorem 1}]
After setting $w=m$, where $m$ is a nonnegative integer and $z \in \C$, $z \notin \Z_{\geq0}$, we can consider multiplying ($\ref{GammaDef}$) by $\frac{ \Gamma (-z)}{\Gamma(-z)}$ and then applying Euler's reflection formula. We get the following equation:
\begin{align*}
  \binom{m}{z} 
&=\frac {\Gamma(m+1)}{\Gamma(z+1) \Gamma(m-z+1)}\frac{ \Gamma (-z)}{\Gamma(-z)}\\&=-
\frac {m!}{\Gamma(z)\Gamma(1-z)}\frac{ \Gamma(-z)}{\Gamma(m-z+1)}=\frac {(-1)^{m}m!\sin(\pi z)}{\pi \prod_{k=0}^{m}(z-k)}.
\end{align*}
By partial fraction decomposition it can be shown that
$$ \frac {(-1)^{m}m!\sin(\pi z)}{\pi \prod_{k=0}^{m}(z-k)}= \frac {\sin(\pi z)}{\pi}\sum_{k=0}^{m}\binom m k \frac {(-1)^k}{z-k}.$$
Since $\cos(\pi k)=(-1)^k$ and $\sin(\pi k)=0$, we may rewrite the right-hand side as
$$\sum_{k=0}^{m}\binom m k \frac {\cos(\pi k)\sin(\pi z)-\sin(\pi k)\cos(\pi z)}{\pi(z-k)}=\sum_{k=0}^{m}\binom m k \sinc(z-k).$$
Furthermore, we can extend the equivalence to the nonnegative integers since \linebreak[4]$\sinc(z-k)$ is $0$ on the integers except for $z=k$ and, hence, for $z\in\Z_{\geq0}$, the sum reduces to $\binom m z$.
\end{proof}

\begin{proof}[Proof of Theorem \ref{Theorem 2}]
Consider the following series of complex functions:
$$\sum_{k=0}^\infty \binom w k  \sinc(z-k).$$
To justify the interchange of sum and integral that we are going to carry out later, we need to establish the absolute convergence of the series. 

Raabe's test (see, for instance, {\cite[p. 39]{Brom}}) states that a series of positive terms $a_k$ converges if
$$ \lim_{k\to\infty} k\left(\frac{a_k}{a_{k+1}}-1\right)>1.$$ 
Setting $a_k=\|\binom w k \sinc(z-k)\|$, we get
$$\frac{a_k}{a_{k+1}}=\left\| \frac{\Gamma(w-k) (k+1)!}{\Gamma(w-k+1)k!} \frac{(z-k-1)}{(z - k)}\right\|=(k+1)\left\| \frac{z-k-1}{(w-k)(z - k)}\right\|.$$
Since $\|a+ib-k\|= \|a-k\| \sqrt{1+\frac{b^2}{(a-k)^2}}$, the limit becomes
$$\lim_{k\to\infty} k\left[\frac{(k+1)(k+1-\mathfrak{R}(z))}{(k-\mathfrak{R}(w))(k-\mathfrak{R}(z))}-1\right]=2+\mathfrak{R}(w).$$ 
It follows that the series is absolutely convergent for $\mathfrak{R}(w)>-1$.

Let us calculate the Fourier transform of the series in Theorem \ref{Theorem 2}: 
$$\sum_{k=0}^{\infty}\binom w k \int_{-\infty}^{\infty}\!\sinc(x-k)e^{- 2\pi i \xi x}\,dx=\rect\left( \xi\right )\sum_{k=0}^{\infty}{\binom {w} k e^{-2\pi i\xi k}}.$$
Using (\ref{BinTheor}), for $\mathfrak{R}(w)>0$ or $w=0$ , we can reduce the last expression to
\begin{equation}\label{FourierT}(1+e^{-2\pi i\xi})^{w}\rect\left( \xi\right ).
\end{equation}
In {\cite[p. 294]{Ram}}, it is proved that, for $\eta \in \R$,
\begin{align*}
&\int_{-\infty}^{\infty} \frac{e^{i\eta x}}{\Gamma(\alpha+x)\Gamma(\beta-x)}\,dx \\
&\qquad = 
\begin{cases}
     \displaystyle \frac{\left[2\cos(\frac{1}{2}\eta)\right]^{\alpha+\beta-2}}{\Gamma(\alpha+\beta-1)} e^{\frac{1}{2}i\eta(\beta-\alpha)}, 
      		& \text{if  $\mathfrak{R}(\alpha+\beta)>1$ and $|\eta|{< \pi}$;} \\
       0, & \text{if  $\mathfrak{R}(\alpha+\beta)>2$ and $|\eta|{\geq \pi}$}.                
  \end{cases}
\end{align*}
For $\alpha=1$, $\beta=w+1$, $\mathfrak{R}(w)>0$ or $w=0$ and $\eta=-2\pi \xi$, this integral can be rewritten as
\begin{align*}
\displaystyle\int_{-\infty}^{\infty} \binom w x e^{-2\pi i\xi x}\,dx &=2^w \cos^w\left(\pi \xi \right)e^{-\pi i\xi w}\rect\left( \xi\right)\\
&=2^w \left(\frac{e^{\pi i\xi}+ e^{-\pi i\xi}}{2}\right)^w e^{-\pi i\xi w}\rect\left( \xi\right) \\
&=(1+e^{-2\pi i\xi})^{w}\rect\left( \xi\right ).
\end{align*}
Thus, by the Fourier inversion theorem, the series in Theorem \ref{Theorem 2} coincides with (\ref{GammaDef}) for $\mathfrak{R}(w)>0$ or $w=0$. Furthermore, since both formulas are analytic for $\mathfrak{R}(w)>-1$,  we can extend the equivalence to the region of convergence of the series, i.e., for $\mathfrak{R}(w)>-1$. 
\end{proof}

\begin{proof}[Proof of Theorem \ref{Theorem 3}]
By Theorem $\ref{Theorem 2}$, we may write
$$\displaystyle \int_{-\infty}^{\infty} \binom w x f(x)\,dx
= \displaystyle\int_{-\infty}^{\infty}\sum_{k=0}^{\infty} \binom w k \sinc(x-k)\,f(x)\,dx.$$\\
The supposition that $f$ is in either $L^1(\R)$ or $L^2(\R)$, together with Theorem \ref{Theorem 2}, implies the absolute convergence of the series and hence allows to interchange the sum and integral:
$$\sum_{k=0}^{\infty} \binom w k \displaystyle\int_{-\infty}^{\infty}f(x)\,\sinc(x-k)\,dx=\sum_{k=0}^{\infty} \binom w k \displaystyle\int_{-\frac{1}{2}}^{\frac{1}{2}}\widehat f(\xi) e^{2\pi i\xi k}\,d\xi.$$
The last step follows from (\ref{Plancherel}).
\end{proof}

\begin{proof}[Proof of Corollary \ref{Corollary}]
By definition, any bandlimited function $g$ with bandwidth $\frac{1}{2}$ can be written as
$$\displaystyle\int_{-\frac{1}{2}}^{\frac{1}{2}}\widehat g(\xi) e^{2\pi i\xi x}\,d\xi.$$
It follows from Theorem $\ref{Theorem 3}$ that
$$\int_{-\infty}^{\infty}\! \binom w x g(x) \,dx=\sum_{k=0}^{\infty}\binom w k \displaystyle \int_{-\frac{1}{2}}^{\frac{1}{2}}\! \widehat g(\xi) e^{2\pi i\xi k} \,d\xi=\sum_{k=0}^{\infty}\binom w k g(k).$$

\end{proof}

\section{Secondary Results.}
The binomial coefficient $\binom w z$ has antiderivative
\begin{equation}\label{Application1}\frac{1}{\pi}\sum_{k=0}^\infty \binom w k  \Si(\pi z -\pi k), \thickspace \thickspace\thickspace\mathfrak{R}(w)>-1.
\end{equation}
\begin{proof}By the definition of $\Si(z)$, we can write
$$\frac{1}{\pi}\sum_{k=0}^\infty \binom w k  \Si(\pi z -\pi k) =\frac{1}{\pi} \sum_{k=0}^\infty \binom w k \int_{0}^{\pi z-\pi k} \frac{\sin t}{t} \,dt.$$
Letting $t=\pi (u-k)$, the right-hand side becomes
$$ \sum_{k=0}^\infty \binom w k \int_{k}^{z}\sinc(u-k)\,du= \int_{k}^{z} \sum_{k=0}^\infty \binom w k \sinc(u-k) \,du.$$
Therefore, to end the proof, it is sufficient to apply Theorem \ref{Theorem 2} to show that the derivative of the last expression coincides with the binomial coefficient $\binom w z$ for $\mathfrak{R}(w)>-1$. Note that the interchange of the sum and integral above is justified by the absolute convergence of the series.
\end{proof}

In addition, Theorem 2 gives the following absolutely convergent series:
\begin{equation}\label{Application2}
\sum_{k=0}^\infty \binom w k \frac{\cos(\pi z-\pi k)}{\pi z-\pi k}= \binom w z \cot(\pi z),  \thickspace \thickspace\thickspace z\in\C,\thickspace z \notin \Z,\thickspace \mathfrak{R}(w)>-1.
\end{equation}
\begin{equation}\label{Application3}
\sum_{k=0}^\infty \binom w k \binom k z \binom z k= \binom w z,  \thickspace \thickspace\thickspace z\in\C, \thickspace\mathfrak{R}(w)>-1.
\end{equation}
\begin{proof}[Proof of (\ref{Application2})]
\begin{align*}\sum_{k=0}^\infty \binom w k \frac{\cos(\pi z-\pi k)}{\pi z-\pi k}&= \frac{\cos(\pi z)}{\pi}\sum_{k=0}^\infty \binom w k \frac{(-1)^k}{z-k}\\
&=\frac{\cos(\pi z)}{\pi} \left[ \frac{\pi}{\sin(\pi z)} \binom w z\right]= \binom w z \cot(\pi z).
\end{align*}
The convergence follows directly from Theorem $\ref{Theorem 2}$.
\end{proof}
\begin{proof}[Proof of (\ref{Application3})] 
Writing the product $\binom k z  \binom z k$  using the  gamma  function, and applying Euler's reflection formula, one obtains
\begin{eqnarray*}    \binom k z \binom z k &=&  \frac {1}{\Gamma(k-z+1)\Gamma(z-k+1)}=\frac {1}{\Gamma(k-z+1)\Gamma(z-k)(z-k)}\\
 &=& \frac {\sin(\pi z-\pi k)}{\pi} \frac{1}{(z-k)}=  \sinc(z-k).
\end{eqnarray*}
 The  claim thus follows  from Theorem \ref{Theorem 2}.
\end{proof}

With the help of Theorems \ref{Theorem 2} and \ref{Theorem 3}, we can find functions whose integrals on $\R$ look almost identical to the series of their values on the integers. For instance, the following two integrals are valid for $\mathfrak{R}(w)>-1$.
\begin{equation}\label{Application4}\displaystyle \int_{-\infty}^{\infty} \binom w x \frac{1}{x+\alpha}\,dx=\sum_{k=0}^\infty \binom w k \frac{1}{k+\alpha}-e^{i\pi \alpha}B(w+1,\alpha), 
\end{equation}
where $\mathfrak{I}(\alpha) >0$ or $\alpha \in \Z_ {\geq 1}$.
\begin{equation}\label{Application5}\displaystyle \int_{-\infty}^{\infty} \binom w x \frac{1}{x^2+\alpha^2}\,dx= \sum_{k=0}^\infty \binom w k \frac{1}{k^2+\alpha^2}-\frac{\pi}{\alpha(e^{2\pi\alpha}-1)}\left[\binom w {i\alpha} + \binom w {-i\alpha} \right],
\end{equation}
where $\mathfrak{R}(\alpha)>0$.

\begin{proof}[Proof of (\ref{Application4})] By Theorem \ref{Theorem 2}, for $\mathfrak{R}(w)>-1$, we can write
\[
\int_{-\infty}^{\infty} \binom w x \frac{1}{x+\alpha}\,dx =\frac{1}{\pi} \displaystyle \int_{-\infty}^{\infty}\! \sum_{k=0}^\infty \binom w k \frac{\sin(\pi x -\pi k)}{(x-k)(x+\alpha)}\,dx.
\]
Setting $\mathfrak{I}(\alpha) >0$ or $\alpha \in \Z_ {\geq 1}$ implies the convergence of the integral. Therefore the series is absolutely convergent and we can interchange the sum and integral:
\begin{align*}
\frac{1}{\pi} \sum_{k=0}^\infty &\binom w k  \int_{-\infty}^{\infty}\!\frac{\sin(\pi x -\pi k)}{(x-k)(x+\alpha)}\,dx \\
&=\frac{1}{\pi} \sum_{k=0}^\infty \binom w k \frac{(-1)^k}{k+\alpha} \left[ \int_{-\infty}^{\infty}\!\frac{\sin(\pi x)}{x-k}\,dx -  \int_{-\infty}^{\infty}\!\frac{\sin(\pi x)}{x+\alpha}\,dx\right]\\
&= \frac{1}{\pi} \sum_{k=0}^\infty \binom w k \frac{(-1)^k}{k+\alpha} \left[ (-1)^k \pi - \pi e^{i\pi\alpha}\right]\\
&=\sum_{k=0}^\infty \binom w k \frac{1}{k+\alpha}-e^{i\pi \alpha}\sum_{k=0}^\infty \binom w k \frac{(-1)^k}{k+\alpha}\\
&=\sum_{k=0}^\infty \binom w k \frac{1}{k+\alpha}-e^{i\pi \alpha}B(w+1,\alpha).
\end{align*}
The last step follows from (\ref{BetaFunc}).
\end{proof}

\begin{proof}[Proof of (\ref{Application5})]
Setting $\mathfrak{R}(\alpha)>0$ implies the absolute convergence of the integral and allows one to use Theorem \ref{Theorem 3}:
\begin{align*}
\displaystyle \int_{-\infty}^{\infty} \binom w x \frac{1}{x^2+\alpha^2}\,dx&= \frac{\pi}{\alpha}\sum_{k=0}^\infty \binom w k \displaystyle \int_{-\frac{1}{2}}^{\frac{1}{2}}\!e^{-2\pi \alpha |\xi|}e^{2\pi i \xi k}\,d\xi\\
&=\sum_{k=0}^\infty \binom w k \frac{1-e^{-\pi \alpha}(-1)^k}{k^2+\alpha^2}\\&= \sum_{k=0}^\infty \binom w k \frac{1}{k^2+\alpha^2}-\frac{\pi }{\alpha(e^{2\pi \alpha}-1)}\left[\binom w {i\alpha} + \binom w {-i\alpha} \right].
\end{align*}
The last step can be derived from Theorem $\ref{Theorem 2}$.
\end{proof}
Moreover, a surprising integral can be found with Theorem \ref{Theorem 3}:
\begin{equation}\label{Application6}
\displaystyle \int_{-\infty}^{\infty} \binom {i\alpha}{x} \sech\left(\frac{\pi x}{\alpha}\right) dx=\frac{\alpha 2^{i\alpha}}{\sqrt{\pi}}\frac{\Gamma(\frac{i\alpha}{2}+\frac{1}{2})}{\Gamma(\frac{i \alpha}{2}+1)},
\end{equation}
where $\mathfrak{R}(\alpha)>0$ and $\mathfrak{I}(\alpha)>-1$.
\begin{proof} 
Assuming $\mathfrak{R}(\alpha)>0$ and $\mathfrak{I}(\alpha)>-1$, we can use Theorem \ref{Theorem 3} to write
$$\displaystyle \int_{-\infty}^{\infty} \binom {i\alpha}{x} \sech\left(\frac{\pi x}{\alpha}\right) dx = \alpha \sum_{k=0}^{\infty}\binom {i\alpha} k  \displaystyle \int_{-\frac{1}{2}}^{\frac{1}{2}}\! \sech \left( \pi \alpha \xi\right) e^{2\pi i \xi k}\,d\xi.$$
Using Euler's formula to exploit the symmetries of the trigonometric functions and the absolute convergence of the integral, we can rewrite the last expression as
 $$2 \alpha  \int_{0}^{\frac{1}{2}}\!  \sum_{k=0}^{\infty}\binom {i\alpha} k \cos(2\pi \xi k )\, \sech ( \pi \alpha \xi) \,d\xi = \alpha 2^{i\alpha+1} \int_{0}^{\frac{1}{2}}\!  \cos^{i\alpha}(\pi\xi)\, d\xi.  $$
Then, upon setting $\sin^2(\pi\xi)=t$, we get 
$$\frac{\alpha 2^{i\alpha}}{\pi} \int_{0}^{1} t^{-1/2}(1-t)^{\frac{i\alpha}{2}-\frac{1}{2}}\,dt=\frac{\alpha 2^{i\alpha}}{\sqrt{\pi}}\frac{\Gamma(\frac{i\alpha}{2}+\frac{1}{2})}{\Gamma(\frac{i \alpha}{2}+1)}.$$
The last step follows from (\ref{BetaFunc2}).
\end{proof}

The last application that we are going to give in the article is an integral representation of the complex binomial coefficient:
\begin{equation}\label{Application7}
\displaystyle \int_{-\infty}^{\infty} \binom w x \sinc(x-z)\, dx = \binom w z, \thickspace \thickspace\thickspace z \in \C,\thickspace \mathfrak{R}(w)>-1.
\end{equation}
\begin{proof}
As follows from Corollary $\ref{Corollary}$, since $\sinc(x)$ is bandlimited to $[-\frac{1}{2},\frac{1}{2}]$, we can write
$$
\displaystyle \int_{-\infty}^{\infty} \binom w x \sinc(x-z) \,dx = \sum_{k=0}^\infty \binom w k  \sinc(k-z)=\binom w z.
$$
\end{proof}
\begin{acknowledgment}\rm The author wishes to thank Francesca Aicardi for her helpful tips on the presentation.
\end{acknowledgment}

lorenzodavid@outlook.it
\vfill\eject

\end{document}